\documentclass[11pt,a4paper]{article}
\date{}
\usepackage{enumerate,amsmath,amsthm,amssymb,latexsym,
amsfonts,amscd}
\usepackage[all]{xy}

\title{Homology theory formulas for generalized Riemann-Hurwitz and generalized monoidal transformations \thanks{Dedicated to Professor Samuel Gitler--Hammer \textit{in memoriam}}}

\author{James  F. Glazebrook
and Alberto Verjovsky
\thanks{
This work was partially supported by PAPIIT (Universidad
Nacional Aut\'onoma de M\'exico) \#IN103914.}}

\theoremstyle{plain}
\newtheorem{lemma}{Lemma}[section]
\newtheorem{proposition}{Proposition}[section]
\newtheorem{theorem}{Theorem}[section]

\theoremstyle{definition}
\newtheorem{definition}{Definition}[section]
\newtheorem{example}{Example}[section]
\newtheorem{remark}{Remark}[section]

\numberwithin{equation}{section}

\newcommand{\BM}{{\rm BM}}

\newcommand{\codim}{{\rm codim}}
\newcommand{\coker}{{\rm coker}}

\newcommand{\Hom}{{\rm Hom}}

\newcommand{\rank}{{\rm rank}}

\newcommand{\SM}{{\rm SM}}

\newcommand{\SO}{{\rm SO}}

\newcommand{\Sp}{{\rm Sp}}

\newcommand{\sign}{{\rm sign}}

\newcommand{\vir}{{\rm vir}}

\renewcommand{\a}{\alpha}
\newcommand{\be}{\beta}

\newcommand{\F}{\mathcal F}

\newcommand{\J}{\mathcal J}
\newcommand{\K}{\mathcal K}

\newcommand{\cL}{\mathcal L}

\newcommand{\cS}{\mathcal S}
\newcommand{\T}{\mathcal T}
\newcommand{\U}{\mathcal U}
\newcommand{\V}{\mathcal V}

\newcommand{\bC}{\mathbb{C}}

\newcommand{\bP}{\mathbb{P}}
\newcommand{\bQ}{\mathbb{Q}}
\newcommand{\bR}{\mathbb{R}}

\newcommand{\bZ}{\mathbb{Z}}

\newcommand{\lra}{\longrightarrow}

\newcommand{\ovsetl}[1]{\overset {#1}{\lra}}

\newcommand{\what}{\widehat}

\newcommand{\ul}{\underline}
\newcommand{\wti}{\widetilde}

\newcommand{\sfL}{\mathsf{L}}

\newcommand{\del}{\partial}

\newcommand{\med}{\medbreak}

\begin{document}

\maketitle

\begin{abstract}
In the context of orientable circuits and subcomplexes of these as representing certain singular spaces, we consider characteristic class formulas generalizing those classical results as seen for the Riemann-Hurwitz formula for regulating the topology of branched covering maps and that for monoidal transformations which include the standard blowing-up process. Here the results are presented as cap product pairings, which will be elements of a suitable homology theory, rather than characteristic numbers as would be the case when taking Kronecker products once Poincar\'{e} duality is defined.
We further consider possible applications and examples including branched covering maps, singular varieties involving virtual tangent bundles, the Chern-Schwartz-MacPherson class, the homology L-class, generalized signature, and the cohomology signature class.
\end{abstract}

\medbreak

\med



\section{Introduction}\label{introduction}

Given a topological group with an associated classifying space, we consider certain characteristic class expressions pertaining to a (virtual) bundle theory over spaces that are oriented circuits (viz. triangulated pseudomanifolds). The techniques follow from general topological constructions as presented in \cite{GGV} (cf. \cite{GV1,GV2}) leading to formulas of a generalized Riemann-Hurwitz type, besides similar expressions that may regulate the topology of generalized monoidal transformations which include the standard blowing-up process.

 The development of ideas in this present paper highlights a basic difference from \cite{GGV,GV1,GV2}, namely, that in contrast to using the usual Kronecker pairing in cohomology/homology that leads to actual characteristic numbers, we instead formulate the results applicable to a suitable homology theory over the circuits in question by taking cap product pairings, once fundamental classes are prescribed. This is a possible approach to a theory of ramified maps, or to a generalized blowing-up process in the presence of singular objects.
 When Poincar\'e duality is defined, it can be used to recover the numerical results via the Kronecker pairing in the usual way. Here we consider several instances in order to show how the general results may be applicable. To this extent, particular interest is when the circuits in question (e.g. branching sets) are stratified pseudomanifolds, or possibly singular complex projective varieties in which, for instance, Chern-Schwartz-MacPherson classes can be implemented, as well as the homology L-class, and the cohomology signature class in other cases.

 We dedicate this paper to the memory of Professor Samuel Gitler--Hammer who was our co-author in \cite{GGV}. Samuel was an exceptionally accomplished and gifted mathematician who greatly inspired the many colleagues and students who were privileged to have worked with him. We can imagine that he would have been very interested in our current project with view to a renewed collaboration. Alas, for us, this will not be the case. His presence will be sadly missed by his family, and by the mathematical community at large.


\section{The topological background}\label{topological}

\subsection{Adapted pairs}\label{adapted}

\begin{definition}
Following \cite{GGV}, let $\Lambda$ be a given (commutative) coefficient ring and $(\Xi, \Sigma)$ a pair of CW-complexes, with $\Xi$ of dimension $n$ and $\Sigma$ a subcomplex of codimension $r \geq 2$, so that $H_q(\Xi, \Lambda) = 0$ for $q > n$, and $H_q(\Sigma, \Lambda) = 0$ for $q \geq n-1$. Then the pair $(\Xi, \Sigma)$ is called $(n, \Lambda)$-\emph{adapted} if $H_n(\Xi, \Lambda) \cong \Lambda$, and the following condition holds:
there exists a neighborhood $B(\Sigma)$ of $\Sigma$, such that $\Sigma$ is a deformation retract of the interior $B^{0}(\Sigma)$ of $B(\Sigma)$, and the quotient map $p_n: C_n(\Xi) \lra C_n(\Xi)/C_n(\Xi - B^{0}(\Sigma))$ induces an isomorphism
\begin{equation}
    p_{*}: H_n(\Xi) \ovsetl{\cong} H_n(\Xi-B^{0}(\Sigma)).
\end{equation}
\end{definition}

Having defined an adapted pair we continue from \cite{GGV} to form the subspace $\K$ of $\Xi \times I$, where
\begin{equation}
\K = (\Xi \times \del I) \cup (\Xi-B^{0}(\Sigma)) \times I,
\end{equation}
together with the double $S(\Sigma) \subset \K$, given by
\begin{equation}
S(\Sigma) = (\del B(\Sigma) \times I) \cup (B(\Sigma) \times \del I).
\end{equation}
Now let
$$
\K_1 = (\Xi \times \{0\}) \cup (\Xi-B^{0}(\Sigma)) \times [0, \frac{3}{4}],
$$
and
$$
\K_2 = (\Xi \times \{1\}) \cup (\Xi-B^{0}(\Sigma)) \times [\frac{1}{4}, 1],
$$
and let $S(\Sigma)_i = S(\Sigma) \cap \K_i$, for $i=1,2$.

By this construction, the spaces $\K_i$ are homotopically equivalent to $\Sigma$, and both the spaces $\K/\K_1$ and $S(\Sigma)/S(\Sigma)_1$ are homotopically equivalent to the generalized Thom space $\Xi /(\Xi - B^{0}(\Sigma))$ (cf. \cite{Th}).
It follows from the cofibration
\begin{equation}
S(\Sigma)_1 \lra S(\Sigma) \lra S(\Sigma)/S(\Sigma)_1
\end{equation}
and the above definition that
\begin{equation}
H_n(S(\Sigma)_1, \Lambda) \cong  H_n(S(\Sigma)/S(\Sigma)_1) \cong \Lambda.
\end{equation}
We assume there is a choice of generators for $H_n(\Xi, \Lambda)$ and $H_n(S(\Sigma), \Lambda)$ giving (by definition) an orientation or a fundamental class $[\Xi]$ of $\Xi$ and $[S(\Sigma)]$ of $S(\Sigma)$, respectively.

\begin{remark}\label{top-example}
Observe that the conditions defining an $(n, \Lambda)$-adapted pair above are immediately satisfied when $\Xi$ is a closed (compact without boundary) connected orientable $n$-manifold, and $\Sigma$ is a closed connected and orientable submanifold of codimension $r \geq 2$, with $\Lambda$
 any coefficient ring. This example also applies to topological, PL, as well as smooth (sub)manifolds, with $S(\Sigma)$ a corresponding normal sphere bundle
 (also closed, connected and orientable for $H_n(S(\Sigma), \Lambda)  \cong \Lambda$, given the above topological type of $\Sigma$).
\end{remark}

Let $G$ be a topological group admitting a classifying space $BG$, and suppose $P \in H^*(BG, \Lambda)$ is a cohomology class.
For a given vector bundle $E$ of rank $\ell$ over $\Xi$, with structure group $G$, we have then a corresponding characteristic class
$P = P_E \in H^*(BG, \Lambda)$, and an element $P(E) \in H^*(\Xi, \Lambda)$
defined by $P(E) = \Phi^*_E(P)$, where $\Phi_E: \Xi \lra BG$ is the classifying map. Typically, $G = \U(\ell), \SO(\ell)$ or $\Sp(\ell)$. Recall that $H^*(BG, \Lambda)$ is a polynomial ring in the Chern classes (for complex vector bundles, with $\Lambda = \bZ$); in the Pontrjagin classes and Euler class (for real oriented vector bundles, with $\Lambda = \bZ[ \frac{1}{2}]$); or the corresponding Pontrjagin classes (for symplectic vector bundles, with $\Lambda = \bZ$), and in the Stiefel-Whitney classes (for real vector bundles, with $\Lambda = \bZ_2$). We refer to \cite{BT,Chern,Hirzebruch} for the basic details.

\subsection{Constructing the double $S(\Sigma)$}\label{double}

Starting from these axioms for an $(n, \Lambda)$-adapted pair above, the `double' $S(\Sigma)$ of $\Sigma$ was constructed by general topological means in \cite{GGV}. Here we will use a special, more geometric, case of this construction along the lines of \cite{Atiyah1,Karoubi}, which to some extent follows the approach given in \cite{Vanque}.

We assume $\Xi$ to be a closed and connected
orientable simple (finite) $n$-circuit (or a \emph{triangulated pseudomanifold} in the sense of \cite{Goresky1}; see also \cite{ES,Lefschetz1}), and $\Sigma$ an arbitrary closed and connected subcomplex of (real) codimension $r \geq 2$, such that $(\Xi, \Sigma)$ satisfies the conditions of an $(n, \Lambda)$-adapted pair.

Here $B(\Sigma)$ is taken to be a closed tubular neighborhood of $\Sigma$. For $i=1,2$, let $B_i(\Sigma)$ be two distinct copies of $B(\Sigma)$, together with restriction maps
\begin{equation}\label{restriction-1}
q_i: B_i(\Sigma) \lra \Sigma.
\end{equation}
Identifying $B_1(\Sigma)$ and $B_2(\Sigma)$ along their common boundary $\del B(\Sigma) = \del B_1(\Sigma) = \del B_2(\Sigma)$, the double $S(\Sigma)$ is formed by setting
\begin{equation}\label{double-1}
S(\Sigma) = B_1(\Sigma) \cup_{\del B(\Sigma)} B_2(\Sigma),
\end{equation}
which leads to an $r$-sphere bundle
\begin{equation}\label{double-2}
q: S(\Sigma) \lra \Sigma,
\end{equation}
for which the fundamental class (or orientation) $[S(\Sigma)]$ is prescribed as the union of cycles
\begin{equation}\label{double-3}
[S(\Sigma)] = [B_1(\Sigma)] \cup (-)[B_2(\Sigma)].
\end{equation}

At this stage, it seems fitting for us to recall the general result of \cite[Theorem 1.1]{GGV} as it was expressed in terms of Kronecker pairings ($\langle~,~\rangle$):

\begin{theorem}\label{main-1}
Suppose $(\Xi, \Sigma)$ is an $(n, \Lambda)$-adapted pair and $E, F$ are $G$-bundles over $\Xi$, such that on $\Xi - \Sigma$ there
exists a homotopy
\begin{equation}
\theta: \Phi_E \vert_{\Xi - \Sigma} \sim \Phi_F \vert_{\Xi  - \Sigma}.
\end{equation}
Then there exists a $G$-bundle $\xi_{\theta} \lra S(\Sigma)$ and orientations $[\Xi]$ and $[S(\Sigma)]$, such that for any class $P \in H^n(BG,\Lambda)$, we have the following equality:
\begin{equation}\label{pairing-1}
\langle P(E) - P(F), [\Xi] \rangle = \langle P(\xi_{\theta}), [S(\Sigma)] \rangle .
\end{equation}
\end{theorem}

\subsection{The clutching construction}\label{clutching}

Let $E,F$ be $G$-bundles over $\Xi$. Specifically, we mean that $E,F \lra \Xi$ are vector bundles with structure group $G$.
Since we are considering general characteristic classes, we could also take $E,F$ to be locally free sheaves. Consider now a homomorphism
$\psi: E \lra F$, such that
\begin{itemize}
\item[i)] $\psi:E \vert_{\Xi - \Sigma} \ovsetl{\cong} F\vert_{\Xi - \Sigma}$

\item[ii)] $\psi\vert_{\Sigma}$ has constant rank.
\end{itemize}
The above data involving $\psi$ induces a `clutching function' $\eta$ which is used to clutch $E$ and $F$ over the double $S(\Sigma)$ as outlined below. We refer to \cite{Atiyah1,Karoubi,Vanque} for details of this construction. We have the following exact sequence of vector bundles on $\Sigma$:
\begin{equation}\label{clutching-1}
0 \lra K_1 \lra E\vert_{\Sigma} \ovsetl{\psi} F \vert_{\Sigma} \lra K_2 \lra 0,
\end{equation}
where $K_1 \cong \ker \psi$,  and $K_2 \cong \coker ~\psi$. Further, let
\begin{equation}\label{ell-def}
L := \psi(E\vert_{\Sigma}) \subset F\vert_{\Sigma}.
\end{equation}
In this construction the bundles $E,F$ are pulled back to $S(\Sigma)$ via $q$, and the clutched bundle $\xi \lra S(\Sigma)$ is specified by
\begin{equation}\label{clutching-2}
\xi= (E, \psi, F) \cong (q_1^*K_1, \eta, q_2^* K_2) \oplus q^*L .
\end{equation}
Letting $K= (q_1^*K_1, \eta, q_2^* K_2)$ we can express \eqref{clutching-2} as
\begin{equation}\label{clutching-3}
\xi = K \oplus q^* L.
\end{equation}
Observe also, that by the above construction, we obtain isomorphisms
\begin{equation}\label{clutching-4}
q_1^* i^* E \cong \xi \vert_{B_1(\Sigma)}, ~\text{and}~ q_2^* i^* F \cong \xi \vert_{B_2(\Sigma)},
\end{equation}
and
\begin{equation}\label{clutching-5}
E \vert_{\Sigma} \cong K_1 \oplus L,~ \text{and}~ F \vert_{\Sigma} \cong K_2 \oplus L.
\end{equation}
For $x \in \Sigma$,
\begin{equation}\label{clutching-6}
K_x = (q_1^* K_1\vert_{B_1(\Sigma)_x}, \eta_x,
q_2^* K_2\vert_{B_2(\Sigma)_x}),
\end{equation}
is the vector bundle over $S(\Sigma)\vert_{x \in \Sigma}$ constructed via the transition function $\eta_x$, seen as the restriction
of $\eta$ to $\del B(\Sigma)_x$. That is, we have an isomorphism

\begin{equation}\label{clutching-7}
\eta_x: q^*_1 K_1 \vert_{\del B(\Sigma)_x} \ovsetl{\cong} q^*_2 K_2
\vert_{\del B(\Sigma)_x}.
\end{equation}

\begin{lemma}\label{clutch-lemma-1}
For $x,y \in \Sigma$ we have $K_x \cong K_y$, and $q_*(\Phi^*_{K_x}(P) \cap [S(\Sigma)_x])$ is a constant multiple $k \mathbf{1}_{\Sigma}$ of the trivial class in $H^*(\Sigma)$.
\end{lemma}

\begin{proof}
Let $c(x,y)$ is a curve in $\Sigma$ joining two points $x,y \in \Sigma$. Then the restrictions $K_1\vert_{c(x,y)}$, and $K_2\vert_{c(x,y)}$ are trivial. This leads to the following diagram in which the vertical maps are isomorphisms
\begin{equation}
\begin{CD}
q_1^*K_1 \vert_{\del B(\Sigma)_x}   @> \eta_x >>  q_2^*K_2 \vert_{\del B(\Sigma)_x} \\
@V \cong VV   @VV \cong V     \\ q_1^*K_1 \vert_{\del B(\Sigma)_y}      @> \eta_y>> q_2^*K_2 \vert_{\del B(\Sigma)_y}
\end{CD}
\end{equation}
and, modulo these isomorphisms, $\eta_x$ and $\eta_y$ are homotopic. Thus $K_x$ and $K_y$, regarded as bundles on $S^r \cong S(\Sigma)_x \cong S(\Sigma)_y$, are isomorphic. Since $\Sigma$ was taken to be connected, this implies that $q_*(\Phi^*_{K_x}(P) \cap [S(\Sigma)_x])$ is independent of $x$, and thus is a constant multiple, $k \mathbf{1}_{\Sigma}$ of the trivial class in $H^*(\Sigma)$.
\end{proof}

\begin{theorem}\label{main-theorem-1}
Let $K$ and $L$ be as in \eqref{clutching-3}, with $q: S(\Sigma) \Sigma$ the $r$-sphere bundle in \eqref{double-2}.
With respect to the clutched bundle
$\xi = (E, \psi, F)$ in \eqref{clutching-2} and the classifying map $\Phi$, suppose: 1) there exists a splitting
$\Phi_{\xi}^*(P) = \Phi^*_{K}(P) \cup \Phi^*_{q^* L}(P)$, and 2) the cohomological degree of $P_K \leq \rank(K) = r$.
Then in $H_*(\Xi, \Lambda)$ we have the following formula
\begin{equation}\label{formula-1}
(P(E) - P(F)) \cap [\Xi] = i_*(P(L) \cup k \mathbf{1}_{\Sigma} \cap [\Sigma]),
\end{equation}
where $k$ is a constant.
\end{theorem}
\begin{proof}
In view of \eqref{clutching-4}, we have
\begin{equation}\label{main-1-1}
(P(E) - P(F)) \cap [\Xi] = i_*q_*(P(\xi) \vert_{B_1(\Sigma)} \cap [B_1(\Sigma)] - P(\xi) \vert_{B_2(\Sigma)} \cap [B_2(\Sigma)]).
\end{equation}
Recalling \eqref{double-3}, we see that \eqref{main-1-1} equals
\begin{equation}\label{main-1-2}
i_*q_*(P(\xi) \vert_{B_1(\Sigma)} \cap [B_1(\Sigma)] + P(\xi) \vert_{B_2(\Sigma)} \cap (-)[B_2(\Sigma)]) = i_*q_*(P(\xi) \cap [S(\Sigma)]).
\end{equation}
Consider the Gysin sequence applied to $q: S(\Sigma) \lra \Sigma$ (see e.g. \cite{BT,Hirzebruch}). We have a long exact sequence:
\begin{equation}\label{gysin}
\ldots H^i(S(\Sigma)) \ovsetl{q_{*}} H^{i-r}(\Sigma) \ovsetl{\cup e} H^{i+1}(\Sigma) \ovsetl{q^*} H^{i+1}(S(\Sigma)) \ldots
\end{equation}
With some abuse of notation $q_*$ in \eqref{gysin} above denotes integration along the fiber, and $\cup e$ denotes cup product with the Euler class.
Let $\a \in \Phi^*_L(P)$, and $\be \in \Phi^*_K(P)$. Then if $q_{*}$ and $q^{*}$ are the maps in \eqref{gysin}, we have via fibre-integration along $S(\Sigma) \vert_{x \in \Sigma}$, the equality $q_{*}(q^*(\a) \cup \be) = \a \cup q_{*}(\be)$. In view of 1) and 2) above, we have then
\begin{equation}\label{main-1-3}
\begin{aligned}
i_* q_*(P(\xi) \cap [S(\Sigma)]) &= i_* q_* (\Phi^*_{\xi}(P) \cap [S(\Sigma)]) \\
&= i_* q_*((\Phi^*_{q^*L}(P) \cup \Phi^*_K(P)) \cap [S(\Sigma)])\\
&= i_*( q_*((\Phi^*_{q^*L}(P) \cup \Phi^*_K(P)) \cap [S(\Sigma)])).
\end{aligned}
\end{equation}
For $x \in \Sigma$, integration over the fiber of \eqref{main-1-3} yields
\begin{equation} \label{main-1-4}
\begin{aligned}
&~ i_*((\Phi^*_L(P) \cap [\Sigma]) \cup q_*(\Phi^*_K(P) \cap [S(\Sigma)_x])) \\
&=
i_*((\Phi^*_L(P) \cap [\Sigma]) \cup q_*(\Phi^*_{K_x}(P) \cap [S(\Sigma)_x])).
\end{aligned}
\end{equation}
Applying Lemma \ref{clutch-lemma-1}, we finally obtain
\begin{equation}
i_*(\Phi^*_L(P) \cup k\mathbf{1}_{\Sigma} \cap [\Sigma]) = i_*(P(L) \cup k \mathbf{1}_{\Sigma} \cap [\Sigma]).
\end{equation}
\end{proof}

\begin{remark}\label{hypothesis-remark}
In applying this result to complex vector bundles (taking $\Lambda = \bZ$), with $P$ corresponding to the total (or top) Chern class $c_{*}$, for instance,  (or $c_{top}$), we see that both assumptions 1) and 2) in Theorem \ref{main-theorem-1} are satisfied (with $r$ the rank of $L$ as a complex vector bundle), so that $\Sigma$ is of real codimension $2r$, with cohomological degree $c_{*} (L) \leq 2r$. Likewise, if $\Phi^* P(~)= e(~)$ is the Euler class of a real oriented vector bundle.
\end{remark}


\section{Examples and applications}

\subsection{A homology generalized Riemann-Hurwitz formula}\label{rh-formula}

In the following we will be considering situations of the following type.
Let $f: (\Xi, \Sigma) \lra (\Xi', \Sigma')$ be a simplicial map of $(n,\Lambda)$-adapted pairs
such that
\begin{itemize}
\item[i)] $f^{-1}(\Sigma') = \Sigma$, and

\item[ii)] $f: \Xi - \Sigma \lra \Xi' - \Sigma'$ is a homeomorphism,
\end{itemize}
(again, $\Sigma \subset \Xi$ and $\Sigma' \subset \Xi'$ are taken to be subcomplexes of $\Xi$ and $\Xi'$ respectively).
Consider then the diagram
\begin{equation}\label{diagram-1}
\begin{CD}
\Sigma @> i>>  \Xi \\
@V g VV   @VV f V     \\ \Sigma'     @> j >> \Xi'
\end{CD}
\end{equation}

\begin{example}\label{rh-smooth}
Let us first see how \eqref{formula-1} looks when the adapted pairs $(\Xi, \Sigma)$ and $(\Xi', \Sigma')$ are smooth manifolds with
$\dim_{\bR} \Xi = \dim_{\bR} \Xi'$, together
with $E = T\Xi$, and $F = f^*T\Xi'$ (and thus $\rank_{\bR} E = \rank_{\bR} F$). Here we take the morphism $\psi$ of \S\ref{clutching} to have constant rank equal to $\dim_{\bR} \Sigma$, with an isomorphism
\begin{equation}\label{isomorphism-1}
\psi: T\Sigma \vert_{\Sigma} \ovsetl{\cong} L = \psi(T\Sigma \vert_{\Sigma}).
\end{equation}
We then obtain from \eqref{formula-1}
\begin{equation}\label{formula-2}
((P(T\Xi) - f^*P(T\Xi')) \cap [\Xi] = i_*((P(T\Sigma) \cup k \mathbf{1}_{\Sigma} \cap [\Sigma]) \in H_*(\Xi).
\end{equation}
\end{example}
Riemann-Hurwitz type formulas in the smooth category expressed in terms of Kronecker pairings with fundamental classes involving Chern, Pontrjagin and Euler classes, for instance, have been obtained in \cite{Brasselet1,GGV,Nair,Vanque,Schwartz1}.
\begin{example}\label{rh-holomorphic}
If, for instance, $\Xi$ and $\Xi'$ are compact complex manifolds with $\dim_{\bC} \Xi = \dim_{\bC} \Xi'=n$, and $f$ is a holomorphic branched covering map with $\deg f = \ell$, having $\Sigma$ as the (complex) codimension $r=1$ branch set, then for top Chern classes one can derive the term (cf. \cite{Brasselet1,Schwartz1,Vanque}) $c_{n-1}(T \Sigma) \cup k \mathbf{1} = (\ell-1) c_{n-1} (T \Sigma)$, and thus
\begin{equation}
(c_n(T\Xi) - f^*c_n(T\Xi')) \cap [\Xi] = i_*((\ell -1) c_{n-1}(T\Sigma) \cap [\Sigma]).
\end{equation}
\end{example}

\subsection{Virtual tangent bundles}\label{virtual-tangent}

When $\Xi$ is a smooth manifold (so the tangent bundle $T\Xi$ exists), then the
cap-product pairing leads
to the \emph{characteristic homology class} $P_*(\Xi) = P(T\Xi) \cap [\Xi]
\in H_*(\Xi, \Lambda)$. But here we are interested in knowing how much of the traditional theory carries through in the singular case.
To this extent we follow, in part, the review article \cite{Schur}.

So to commence, let us take the category of (possibly singular) projective algebraic
varieties. Let $\Xi$ be a singular variety. Then there is the
problem of defining suitable characteristic classes. But suppose $\Xi$ is
realized as a local complete intersection in a smooth variety $M$, so
that the closed inclusion $\ul{j}: \Xi \lra M$ is a regular
embedding. In this case, the \emph{normal cone} $N_{\Xi}M \lra \Xi$ is
a vector bundle over $\Xi$, and hence one can define the \emph{virtual tangent bundle} of $\Xi$ by
\begin{equation}\label{virtual-1}
T_{\vir}\Xi = [\ul{j}^* TM - N_{\Xi}M] \in K^0(\Xi),
\end{equation}
which is independent of the embedding and thus produces a well-defined
element in the Grothendieck group $K^0(\Xi)$ of vector bundles on
$\Xi$ (see \cite{Fulton2}).

In terms of the characteristic cohomology classes of vector bundles as we have already considered, an \emph{intrinsic homology class} (independent of the
embedding) applies relative to the virtual tangent bundle of $\Xi$, and is
defined as
\begin{equation}\label{virtual-2}
P_*^{\vir} (\Xi) := P(T_{\vir}\Xi) \cap [\Xi] \in H_*(\Xi, \Lambda).
\end{equation}
In this case, $[\Xi] \in H_*(\Xi)$ can be taken as the fundamental class (or the class of the structure sheaf) of $\Xi$ in
\begin{equation}
H_*(\Xi) = \begin{cases} ~H^{\BM}_{2*} (\Xi), ~&\text{the Borel--Moore homology in even degrees},\\
CH_*(\Xi), ~&\text{the Chow group}, \\
G_0(\Xi), ~&\text{the Grothendieck group of coherent sheaves}.
\end{cases}
\end{equation}
In terms of the Gysin homomorphism $\ul{j}^{!}: H_*(M) \lra H_{*-d}(\Xi)$
(where $d$ is the embedding codimension), we also have, in relationship to characteristic classes, that
\begin{equation}\label{virtual-3}
\ul{j}^{!}(P(M)) = \ul{j}^{!}(P(TM) \cap [M]) = P(N_{\Xi}M) \cap P_{*}^{\vir}(\Xi).
\end{equation}

\begin{example}\label{singular-ex-1}
If $\Xi$ and $\Xi'$ are singular varieties admitting closed regular embeddings as above, then Chern classes for virtual tangent bundles can be defined (see e.g. the lecture in \cite{Suwa}). In the case that $\Xi$ and $\Xi'$ are equidimensional, and $\Sigma$ a subvariety, each with fundamental classes defined, then from \eqref{isomorphism-1}, we have
\begin{equation}\label{isomorphism-2}
\psi: T_{\vir}\Sigma \vert_{\Sigma} \ovsetl{\cong} L = \psi(T_{\vir}\Sigma \vert_{\Sigma}).
\end{equation}
Given a rational map $f: \Xi \lra \Xi'$ in the setting of \S\ref{rh-formula}, and applying \eqref{formula-1}, we have for total Chern classes
\begin{equation}
(c(T_{\vir} \Xi) - f^*c(T_{\vir} \Xi')) \cap [\Xi] = i_*(c(T_{\vir} \Sigma) \cup k\mathbf{1}_{\Sigma}  \cap [\Sigma]).
\end{equation}
\end{example}

\subsection{Stratified pseudomanifolds}\label{strat}

To commence, we recall the definition as given in \cite{Friedman1,Goresky1}.
\begin{definition}\label{strat-def}
An \emph{$n$-dimensional PL-stratified pseudomanifold $\cS$} is a piecewise linear space (having a compatible family of triangulations) that also possesses a filtration by closed PL-subspaces
\begin{equation}\label{strat-1}
\cS = \cS^n \supset \cS^{n-2} \supset \cS^{n-3} \supset \cdots \supset \cS^1 \supset \cS^0 \supset \cS^{-1} = \emptyset,
\end{equation}
forming a stratification, that satisfies the following properties:
\begin{itemize}
\item[(1)] $\cS - \cS^{n-2}$ is dense in $\cS$;

\item[(2)] for each $k \geq 2$, $\cS^{n-k} - \cS^{n-k-1}$ is either empty or is an $(n-k)$-dimensional PL-manifold;

\item[(3)] if $x \in \cS^{n-k} - \cS^{n-k-1}$, then $x$ has a distinguished neighborhood that is PL-homeomorphic to $\bR^{n-k} \times c\sfL$, where
$c\sfL$ denotes the open cone on a compact $(k-1)$-dimensional manifold $\sfL$ whose stratification is compatible with that of $\cS$.
\end{itemize}
\end{definition}
A PL-stratified pseudomanifold $\cS$ is oriented (or orientable) if  $\cS - \cS^{n-2}$ has that same property. The sets $\cS^i$ are called the \emph{skeleta} (it can be verified from (2) above that each has dimension $i$ as a PL-complex).  The sets $\cS_i = \cS - \cS^{i-1}$ are called the \emph{strata}. In particular, $\cS - \cS^{n-2}$ are called \emph{regular strata}, and the rest are called \emph{singular strata}.
The space $\sfL$ is called the \emph{link} of $x$, or of the stratum containing $x$. We refer to \cite{Friedman1,Goresky1} for the additional topological characteristics of this sort of pseudomanifold.

\subsection{Maps with branch-like singularities}\label{branch}

 In the context of \S\ref{rh-formula} and \S\ref{strat}, let $f: \Xi \lra \Xi'$ be a smooth map of compact oriented manifolds of equal dimension $n$, with degree $\deg f = \nu_f$. Let $\Sigma$ be a closed, connected PL-stratified pseudomanifold of codimension $2$ in $\Xi$ such that $(\Xi, \Sigma)$ is a $(n, \Lambda)$-adapted pair. We say that $f$ has \emph{branch-like singularities} when the following holds. Let $M$ be a closed oriented $n$-cycle in $\Xi$, such that each $B^{(\nu_i)}:= M \cap \cS^i$ is a smooth connected oriented PL-submanifold of $\Sigma$ on which $f$ has local degree $\nu_i = \vert \deg (f \vert B^{(\nu_i)}) \vert \leq \nu_f$. Furthermore, $M$ is assumed to have empty intersection with the singular strata of $\Sigma$. Note that this construction yields a nested sequence of the $B^{(\nu_i)}$ in accordance with the filtration in \eqref{strat-1}. Moreover, in this setting there is some scope in applying \eqref{formula-1} and \eqref{formula-2} to various characteristic classes. We will present such an application in \S\ref{hirzebruch} below.

\subsection{Total Hirzebruch $\cL$-polynomial and the signature}\label{hirzebruch}

Let $f: \Xi \lra \Xi'$ be a smooth map of compact oriented manifolds of equal dimension $4n$ with $f$ having branch-like singularities on $\Sigma$ as described above, such that $(\Xi, \Sigma)$ is a $(4n, \bQ)$-adapted pair. Recalling the above details, with $M$ a closed oriented $4n$-cycle in $\Xi$, it is possible to obtain from \eqref{formula-1} and \eqref{formula-2} several types of formulas in the Pontrjagin classes of the manifolds in question, whenever these classes can be defined.

Here we recall the total Hirzebruch $\cL$-polynomial \cite{Hirzebruch} on setting $P(~)$ in \eqref{formula-2} by $P(\Xi) = \cL(p_1(\Xi), \ldots, p_n(\Xi))$, where $p_i(\Xi)$ are the Pontrjagin classes of $\Xi$, etc.,  to obtain from \eqref{formula-2}:
\begin{equation}\label{formula-3}
(P(\Xi) - f^*P(\Xi')) \cap [\Xi] = i_*( \sum_{\nu_i} P(B^{(\nu_i)}) \cup k \mathbf{1}_{B^{(\nu_i)}} \cap [B^{(\nu_i)}]).
\end{equation}
When Poincar\'e duality is defined, expressions such as \eqref{formula-2} can be expressed in a numerical form via Kronecker pairings. Thus \eqref{formula-3} becomes for the signature $\sigma$:
\begin{equation}\label{formula-4}
\sigma(\Xi) - \nu_f \sigma(\Xi') = k  \sum_{\nu_i} \sigma(B^{(\nu_i)}),
\end{equation}
where $k$ is a suitable constant.

Note that given a smooth (or PL-locally flat) embedding $i: \Sigma \lra \Xi$, the codimension 2 stratified pseudomanifold $\Sigma$ here admits a \emph{homology $L$-class} in the sense of \cite{CS1,Goresky1} given by
\begin{equation}\label{formula-5}
L(\Sigma) = [\Sigma] \cap i^* \cL (P(\Xi) \cup (1 + \eta^2)^{-1}) \in H_{2*} (\Sigma, \bQ),
\end{equation}
where $P(\Xi) \in H^*(\Xi)$ is the total Pontrjagin class of $\Xi$ and $\eta$ is the Poincar\'e dual of $i_*[\Sigma]$ (see also \S\ref{intersect-1}).
When Poincar\'{e} duality is defined, the latter transforms $L(\Sigma)$ to $\mathcal{L}$ as above.

\begin{remark}
Similar signature formulas have been obtained in \cite{Hirzebruch1,Viro}, and noting that the right-hand side of
\eqref{formula-3} is effectively a residual quantity, see also \cite{Nair}.

In the case of branched coverings of $S^4$ by 4-dimensional closed oriented PL-manifolds (reviewed in \cite{IP}), node singularities of the branching set are removable by suitable cobordisms \cite{IP}. Our approach is different since singularities in the strata are avoided by taking the appropriate intersection with the 4-cycle $M$ in order to create the $B^{(\nu_i)}$, as described above.
\end{remark}

\subsection{Chern-Schwartz-MacPherson classes}\label{csm}

Firstly, we recall the \emph{Chern-MacPherson transformation} (over the field $\bC$) \cite{MacPh1}. To an extent we follow the exposition in \cite{Ohmoto}. For a quasi-projective variety $X$ and proper morphisms this is a natural transformation, when $X$ is smooth, from the constructible function functor to the Chow group functor
\begin{equation}\label{csm-1}
c_*: \F_*(X) \lra CH_*(X),
\end{equation}
satisfying the normalization property
\begin{equation}\label{csm-2}
c_*(\mathbb{I}_X) \lra c(TX) \cap [X] \in CH_*(X).
\end{equation}
On the other hand, if $\Sigma'$ is a possibly singular variety, then $c_*(\Sigma') = c_*(\mathbb{I}_{\Sigma'})$ defines the \emph{Chern-Schwartz-MacPherson class} of $\Sigma'$ in $CH_*(\Sigma')$ \cite{MacPh1}.

There is a way this latter class can be realized as in
\cite{Aluffi2}. Consider a closed embedding $j: \Sigma' \lra X$,
still assuming $X$ is smooth. Resolving singularities, one obtains a birational map $f: \wti{X} \lra X$ such that $\wti{X}$ is smooth, with
$\bar{\Sigma}' := f^{-1}(\Sigma)$. Furthermore,
$\Sigma = \wti{X} - \Sigma'$ is a normal crossing divisor in $X$ with smooth irreducible components $D_1, \ldots, D_k$. By induction on $k$, and properties of $c_*$, it is shown in \cite{Aluffi2} that
\begin{equation}\label{csm-3}
c_*(\mathbb{I}_{\Sigma'}) = f_*\big( \frac{c(T\wti{X})}{\Pi(1+D_i)} \cap [\wti{X}] \big) \in CH_*(X).
\end{equation}
In this last expression, the term $c(T\wti{X})/\Pi(1+D_i)$ is shown to equal the total Chern class $c(E^*)$ where
$E = \Omega_{\wti{X}}^1(\log \Sigma)$ is a locally free sheaf of complex differential 1-forms with logarithmic poles along $\Sigma$ having
$\rank_{\bC} E = \dim_{\bC} \wti{X}$.

\begin{example}\label{log-poles}
Within the setting of \S\ref{rh-formula}, we set $\Xi= \wti{X}$ and $\Xi' = X$. We take $E = \Omega_{\wti{X}}^1(\log \Sigma)$, and set $F = f^*(TX)$. Again we have $L= \psi(E\vert_{\Sigma}) \subset F\vert_{\Sigma}$. Setting $P$ as the total Chern class, and then applying the push-forward $f_*$ to \eqref{formula-1}, we have by \cite[Theorem 1]{Aluffi2}:
\begin{equation}\label{log-poles-1}
\begin{aligned}
~& f_*((c(\Omega_{\wti{X}}^1(\log \Sigma)) - f^*c(TX)) \cap [\wti{X}])\\
& = f_* i_* (c(L) \cup k \mathbf{1}_{\Sigma} \cap [\Sigma]) \\
&= - j_* c_{\SM} (\Sigma') \in CH_*(X),
\end{aligned}
\end{equation}
which gives a further realization of the Chern-Schwartz-MacPherson class in the Chow group $CH_*(X)$.
\end{example}


\subsection{Generalized monoidal tranformations for oriented circuits}\label{gen-circuits}

We return now to the generalized set-up of \S\ref{adapted} and
\S\ref{rh-formula} with regards a simplicial map $f: (\Xi, \Sigma) \lra (\Xi', \Sigma')$ satisfying conditions i) and ii) in \S\ref{rh-formula}.
The hypotheses outlined in \S\ref{rh-formula}, gave rise to the concept of a \emph{generalized monoidal transformation} in \cite{GV1}. So far, the development of ideas motivates the following proposition:

\begin{proposition}\label{thom-prop-1}
Consider an $(n,\Lambda)$-adapted pair $(\Xi, \Sigma)$ as in \S\ref{rh-formula} satisfying conditions i) and ii). Let the maps $i,j$ in \eqref{diagram-1} be embeddings such that the Thom spaces $\T(\Sigma, \Xi), \T(\Sigma', \Xi')$ of the normal bundles $\what{N}, N$, of $i,j$ respectively, are defined.
Then for suitable $P \in H^*(BG, \Lambda)$, we have
\begin{equation}\label{mt-formula-1}
(P(\Xi) - f^* P(\Xi')) \cap [\Xi] = \hat{t}^*(u \cdot a) \cap [\Xi] \in H_*(\Xi),
\end{equation}
where $u$ is the Thom class of $\T(\Sigma, \Xi)$, $a \in H^*(\Sigma)$ is some suitable class, and $\hat{t}: \Xi \lra \T(\Sigma, \Xi)$ is the Thom-Pontrjagin map.
\end{proposition}

\begin{proof}
In view of the hypotheses on $i,j$, we have (following e.g. \cite{Gitler1}) a natural commutative diagram
\begin{equation}\label{diagram-2}
\begin{CD}
\Sigma @> \hat{k} >> \what{N} @> i>>  \Xi @> \hat{t}>> \T(\Sigma, \Xi)\\
@V g VV   @V \hat{g} VV @VV f V  @VV h V   \\ \Sigma'  @> k>>  N  @> j>> \Xi' @> t >> \T(\Sigma', \Xi')
\end{CD}
\end{equation}
where $t, \hat{t}$ denote the Thom-Pontrjagin maps, together with excision relations:
\begin{equation}\label{excision-1}
\begin{aligned}
\Xi &= (\Xi- \Sigma) \cup \what{N}, ~ ~  ~  ~ (\Xi - \Sigma) \cap \what{N} = \what{N} - \Sigma,  \\
\Xi' &= (\Xi' - \Sigma') \cup N, ~ ~ (\Xi' - \Sigma') \cap N = N - \Sigma'.
\end{aligned}
\end{equation}
Note that by excision, we have  $H^*(\Xi, \Xi - \Sigma) \cong H^*(\T(\Sigma, \Xi))$. The proposition follows from the topological properties of
$\T(\Sigma, \Xi)$ and the Thom class $u$ of $\T(\Sigma, \Xi)$ relative to pullback via $\hat{t}$ in the far right square of \eqref{diagram-2}.
\end{proof}
In the case of the standard blowing-up process for non-singular varieties (see \S\ref{bu-nonsing} below), the `suitable class' $a \in H^*(\Sigma)$ was explicitly determined, as was the right-hand side of \eqref{mt-formula-1} in terms of total Chern classes in \cite[Theorem (4.3)]{Gitler1}, as exhibited in \eqref{bu-formula-2} below  (cf. \cite{Lascu1,Lascu2,Porteous1}).
\begin{remark}
Once virtual tangent bundles are defined, then following \eqref{formula-1} for connected $n$-circuits in the setting of \S\ref{adapted}, a generalized monoidal transformation  satisfies
\begin{equation}\label{mt-formula-3}
(P(T_{\vir} \Xi) - P(f^*T_{\vir} \Xi')) \cap [\Xi] = i_*(P(T_{\vir} \Sigma) \cup k \mathbf{1}_{\Sigma} \cap [\Sigma]),
\end{equation}
for the appropriate $P \in H^*(BG, \Lambda)$.
\end{remark}

\subsection{Intersection cohomology signature}\label{intersect-1}

Following \cite{Schur}, in the compact case, we have a signature
\begin{equation}\label{sign-1}
\sign(\Xi) = \deg(L_*(\Xi)), ~ \text{with}~ L_*: \Omega(\Xi) \lra H_{2*}(\Xi, \bQ),
\end{equation}
where the latter is the \emph{homology $L$-class transformation} of \cite{CS1} (see also the review of related topics in \cite{Banagl1}). Here $\Omega(\Xi)$ denotes the abelian group of corbordism classes of selfdual constructible complexes. From this it follows that $L_*(\Xi)= L_*([ \J C_{\Xi}])$ is the \emph{homology $L$-class} of \cite{Goresky1} having a distinguished element $\mathbf{1}_{\Xi} = [\J C_{\Xi}]$ the class of their intersection homology complex. Thus $\sign(\Xi)$ is called the \emph{intersection cohomology signature of $\Xi$}.

\begin{proposition}
If $f: (\Xi, \Sigma) \lra (\Xi', \Sigma')$ in \S\ref{rh-formula} is a map of compact spaces, then we have
\begin{equation}
f_*(L_*(\Xi) - L_* (\Sigma)) = L(\Xi') \in H_{2*}(\Xi', \bQ).
\end{equation}
In particular, on taking degrees we have the intersection cohomology signature relationship
\begin{equation}
\sign(\Xi)= \sign(\Xi') + \sign(\Sigma).
\end{equation}
\end{proposition}

\begin{proof}
Let $\U^0$ and $\V^0$ be open tubular neighborhoods of $\Sigma$ and $\Sigma'$, respectively, so that $\V^0 = f(\U^0)$, and let $V_1 = \Xi \times [0,1], V_2 = \Xi' \times [0,1]$. Then set $S = V_1 \wedge V_2/ \sim$, where the identification `$\sim$' is given by
\begin{equation}
(x,0) \in (\Xi' \times \{0\}) - (\V^0 \times \{0\}) \sim (f(x), 1) \in (\Xi \times \{1\}) - (\U^0 \times \{1\}).
\end{equation}
Then, by smoothing the corners created by identification $S$ is an
oriented manifold whose oriented boundary is diffeomorphic to $\Xi' +
\Sigma - \Xi$, and $\Xi$ is oriented cobordant to $\Xi' + \Sigma$
\cite[Proposition (2.2)]{GV1}. From this additivity property of cobordism invariants the first statement follows. The second statement then follows by
taking degrees
 (cf. \cite[Corollary (2.3)]{GV1}).
\end{proof}

 \begin{remark}
In \cite{GV1} the role of $\Sigma$ was there denoted by $W$, and it was aptly entitled a \emph{residual circuit}.
\end{remark}

\subsection{The standard blowing-up process}\label{bu-nonsing}

In the following it will be useful to adapt \eqref{diagram-1} using
more familiar information as
follows. We set $\Sigma' = Y$ and $\Xi' = X$, and consider the inclusion $Y
\subseteq X$ of non-singular (smooth) varieties. Let $\Xi =
\wti{X}$ denote the blow-up of $X$ along $Y$, with $\Sigma = \wti{Y}$:
\begin{equation}\label{diagram-3}
\begin{CD}
\wti{Y} @> i>>  \wti{X} \\
@V g VV   @VV f V     \\ Y      @> j >> X
\end{CD}
\end{equation}
We have then a biholomorphism $\hat{f}: \wti{X} - \wti{Y} \lra X - Y$.
In this situation $\wti{Y} = \mathbb{P}(N_YX \oplus \mathbf{1})$ is the projectivized
normal bundle of $Y$ in $X$, that is a locally trivial bundle over $Y$ with fiber $\bC P^r$ ($r = \codim_{\bC} Y$).
Using \cite[Theorem 15.4]{Fulton2}, we
have
\begin{equation}\label{bu-formula-1}
\begin{aligned}
(c(T\wti{X}) - f^* c(TX)) \cap [\wti{X}] &= i_*(g^* c(TY) \cup \lambda \cap
[\wti{Y} ]) \\
&= i_* (c(L) \cup k\mathbf{1}_{\wti{TY}} \cap [\wti{Y}]).
\end{aligned}
\end{equation}

Completely determining the right hand-side of \eqref{bu-formula-1} essentially leads to the Todd-Segre formula for Chern classes relative to the Chow ring of rational equivalence classes of cycles, as was first established by Porteous \cite{Porteous1,Porteous2}. Since then there have been several variations on this theme, exhibiting alternative proofs along with certain generalizations. We refer to e.g. \cite{Aluffi1,Lascu1,Lascu2,Gitler1} for the particular details, but for now, we briefly expose one common line of approach in terms of the geometry of projective bundles.

Suppose $V$ is a rank $r$ complex vector bundle, and consider the projective bundle $p: \bP(V) \lra Y$. There is a short exact sequence
\begin{equation}\label{proj-1}
0 \lra \cL \lra p^*V \lra Q \lra 0,
\end{equation}
where $\cL$ is the canonical line bundle and $Q$ is the quotient bundle. We have then
\begin{equation}
T\bP(V) = p^* TY \oplus B(T\bC P^{r-1}),
\end{equation}
where $B(T\bC P^{r-1})$ denote the bundle along the fibers in \eqref{proj-1}, which in this case is given by
\begin{equation}\label{proj-2}
B(T\bC P^{r-1}) \cong \Hom(\cL, Q) \cong \cL^* \otimes Q,
\end{equation}
so that
\begin{equation}\label{proj-3}
T \bP(V) = p^* TY \oplus \cL^* \otimes Q.
\end{equation}
Applying these considerations to the short exact sequence
\begin{equation}\label{proj-4}
0 \lra \what{N} \lra f^* N \lra Q \lra 0,
\end{equation}
leads, as in \cite{Gitler1} (cf. \cite{Lascu1,Lascu2,Porteous2}), to a realization of the right-hand side of \eqref{bu-formula-1} given by
\begin{equation}\label{bu-formula-2}
\begin{aligned}
(c(T\wti{X}) - f^* c(TX)) \cap [\wti{X}] &= i_*(g^* c(TY) \cup \lambda \cap
[\wti{Y} ]) \\
&= t^*\big( u(\hat{f}^* c(TY) \cup c(\what{N}) \cup B(\what{N}^*,Q)\big) \cap [\wti{X}].
\end{aligned}
\end{equation}

\bigbreak
\noindent
\textbf{Acknowledgement}
We sincerely thank an anonymous referee for encouraging comments, as well as listing a number of proposed corrections which greatly helped in revising an earlier version of this paper.


James F. Glazebrook.\\
 Department of Mathematics and Computer
Science \\
 Eastern Illinois University \\
600  Lincoln Ave., Charleston, IL 61920--3099, USA \\
jfglazebrook@eiu.edu
\\ (Adjunct Faculty)
\\ Department of Mathematics \\ University of Illinois at
Urbana--Champaign\\ Urbana, IL 61801, USA\\

Alberto Verjovsky\\Instituto de Matem\'{a}ticas\\
Universidad Aut\'{o}noma de M\'{e}xico\\
Av. Universidad s/n, Lomas de Chamilpa\\
Cuernavaca CP 62210, Morelos, Mexico\\
alberto@matcuer.unam.mx


\begin{thebibliography}{99}


\bibitem{Aluffi1}
Aluffi, P.: Chern classes of blow-ups. {Math. Proc. Cambridge Philos. Soc.} \textbf{148}(2), 227–-242 (2010)

\bibitem{Aluffi2}
Aluffi, P.: Differential forms with logarithmic poles and Chern-Schwartz-MacPherson classes of singular vatieties. {C. R. Acad. Sci. Paris S\'{e}r. 1 Math.} \textbf{329}(7), 619--624 (1999)

\bibitem{Atiyah1}
Atiyah, M. F.: {K-Theory}. Benjamin, New York (1967)

\bibitem{Banagl1}
Banagl, M.: The signature of singular spaces and its
refinements to generalized homology theories. In: Friedman, G. et al. (eds.){Topology of Stratified Spaces}, pp. 223--247. MSRI Publ. \textbf{58} Berkeley CA (2011)

\bibitem{BT}
Bott, R., Tu, L. W.: {Differential Forms in Algebraic Topology}. Grad. Texts in Math. \textbf{82}. Springer-Verlag, New York Heidelberg Berlin (1982)

\bibitem{Brasselet1}
Brasselet, J.-P.: Sur une formule de M. H. Schwartz relative aux rev\^{e}tements ramifi\'{e}s. {C. R. Acad. Sci. Paris S\'{e}r. A-B} \textbf{283} (2), A41--A44 (1976)

\bibitem{CS1}
Cappell, S., Shaneson, J.: Stratifiable maps and topological invariants. {J. Amer. Math. Soc.} \textbf{4}, 521--551 (1991)

\bibitem{Chern}
Chern, S. S.: {Complex Manifolds without Potential Theory}. Springer-Verlag, Berlin New-York (1979)


\bibitem{ES}
Eilenberg, S., Steenrod, N.: {Foundations of Algebraic Topology}. Princeton Mathematical Series \textbf{15} Princeton Univ. Press, Princeton NJ (1952)

\bibitem{Friedman1}
Friedman, G.: An introduction to intersection homology with general perversity functions. In: Friedman, G. et al. (eds.) {Topology of Stratified Spaces}, pp. 177--222. MSRI Publ. \textbf{58} Berkeley CA (2011)


\bibitem{Fulton2}
Fulton, W.: {Intersection Theory} (2nd Ed.). Ergebnisse der
Mathematik und ihrer Grenzgebiete \textbf{3}. Folge, Springer-Verlag Berlin (1998)

\bibitem{GGV}
Gitler, S., Glazebrook, J. F., Verjovsky, A.: On the generalized Riemann-Hurwitz formula. {Boletin de la Sociedad Matem\'{a}tica Mexicana}
\textbf{30}(1), 1--11 (1985)

\bibitem{Gitler1}
Gitler, S.: The cohomology of blow ups.
{Boletin de la Sociedad Matem\'{a}tica Mexicana} \textbf{37}(1-2) (Homenaje A Jos\'{e} Adem), 167–-175 (1992)

\bibitem{GV1}
Glazebrook, J. F., Verjovsky, A.: Residual circuits in generalized monoidal transformations.
{Boletin de la Sociedad Matem\'{a}tica Mexicana}
\textbf{33}(1), 19--25 (1988)

\bibitem{GV2}
Glazebrook. J. F., A. Verjovsky, A.: Rational and iterated maps, degeneracy loci, and the generalized Riemann-Hurwitz formula. In: Cisneros-Molina, J. L. et al. (eds.) {Singularities in Geometry, Topology, Foliations and Dynamics}. {Trends in Mathematics} pp. 105--124. Birkh\"{a}user-Springer Internatl. Publ. Switzerland (2017)

\bibitem{Goresky1}
Goresky. M., MacPherson. R.: Intersection homology
theory. {Topology} \textbf{19}, 135--162 (1980)

\bibitem{Hirzebruch}
Hirzebruch, F.: {Topological Methods in Algebraic Geometry}. Grundlehren \textbf{131} 3rd Ed., Springer Verlag (1966)

\bibitem{Hirzebruch1}
Hirzebruch, F.: The signature of ramified coverings. In: Spencer, D. C. and Iyanaga, S. (eds.) {Global Analysis}, pp. 253--265 Univ. Tokyo Press, Tokyo (1969)

\bibitem{IP}
Iori, M., Piergallini, R.: 4-manifolds as covers of the 4-sphere branched over non-singular surfaces. {Geometry and Topology} \textbf{6}, 393--401 (2002)

\bibitem{Karoubi}
Karoubi, M.: {K-Theory: An Introduction}. Grundlehren der
Matematische Wissenschaften \textbf{226} Springer-Verlag,
Berlin-Heidelberg-New York (1978)

\bibitem{Lascu1}
Lascu, A. T., Scott, D. B.: A simple proof of the formula for the
blowing up of Chern classes. {Amer. J. Math.} \textbf{100} (2), 293--301 (1978)


\bibitem{Lascu2}
Lascu, A. T., Mumford, D., Scott, D. B.: The self-intersection
formula and the `formule-clef'. {Math. Proc. Cambridge Philos. Soc.} \textbf{78}, 117-123 (1975)

\bibitem{Lefschetz1}
Lefschetz, S.: {Algebraic Topology}. American Math. Soc., Providence RI (1942)

\bibitem{MacPh1}
MacPherson, R. D.: Chern classes for singular algebraic varieties. {Ann. of Math.} \textbf{100}(2), 423--432 (1974)

\bibitem{Nair}
Nair, S.: Geometric residue theorems for bundle maps. {Comm. Anal. Geom.} \textbf{7}(30), 583--608 (1999)

\bibitem{Vanque}
Ng\^{o} Van Qu\^{e}: Generalisation de la formula de Riemann-Hurwitz. {Canadian J. Math.} \textbf{24}(5), 761--767 (1972)

\bibitem{Ohmoto}
Ohmoto, T.: A note on the Chern-Schwartz-Macpherson class. In: Blanl\oe il, V., and Ohmoto, T. (eds.) {Singularities in Geometry and Topology --Strasbourg 2009}. {IRMA Lectures in Mathematics and Theoretical Physics} \textbf{20}, pp. 117--131. European Math. Soc., Z\"{u}rich Switzerland, (2012)

\bibitem{Porteous1}
Porteous, I. R.: Blowing up Chern classes. {Proc. Camb. Phil. Soc.} \textbf{56}, 118--124 (1960)

\bibitem{Porteous2}
Porteous, I. R.: Todd's canonical classes.
Proceedings of Liverpool Singularities Symposium I, (1969/70), pp. 308--312. {Lecture
Notes in Math.} \textbf{192}. Springer, Berlin (1971)

\bibitem{Schur}
Sch\"{u}rmann, J.: Nearby cycles and characteristic classes of singular spaces. In: V. Blanl\oe il, V., and Ohmoto, T. (eds.) {Singularities in Geometry and Topology --Strasbourg 2009}. {IRMA Lectures in Mathematics and Theoretical Physics} \textbf{20}, pp. 181--205. European Math. Soc., Z\"{u}rich  Switzerland (2012).

\bibitem{Schwartz1}
Schwartz, M.-H.: Champs de rep\`{e}res tangents \`{a} une vari\'{e}t\'{e} presque complexe.
{Bull. Soc. Math. Belg.} \textbf{19}, 389–-420 (1967)

\bibitem{Suwa}
Suwa, T.: Residues of singular holomorphic distributions. In: Blanl\oe il, V., and Ohmoto, T. (eds.) {Singularities in Geometry and Topology --Strasbourg 2009}. {IRMA Lectures in Mathematics and Theoretical Physics} \textbf{20}, pp. 207--247. European Math. Soc., Z\"{u}rich Switzerland (2012).

\bibitem{Th}
Thom, R.: Les singularit\'{e}s des applications diff\'{e}rentiables. {Ann. Inst. Fourie}r \textbf{6}, 43--87 (1955-56)

\bibitem{Viro}
Viro, O., Ya.: The signature of a branched covering. (Russian) {Mat. Zametki} \textbf{36}(4), 549--557 (1984). (English translation: {Math. Notes} \textbf{36}(3-4), 772--776)(1984))

\end{thebibliography}
\end{document}